\documentclass[12pt, oneside, reqno]{amsart}
\usepackage{amssymb}
\usepackage{amsthm}
\usepackage{mathtools}
\usepackage[normalem]{ulem}
\usepackage{amsmath, geometry}
\usepackage{listings}
\usepackage[utf8]{inputenc}
\usepackage[english]{babel}
\usepackage[dvipsnames]{xcolor}
\usepackage{amsfonts}
\usepackage{lipsum}
\usepackage{subcaption}

\usepackage{tikz}
\usetikzlibrary{cd}
\usepackage{mathrsfs}
\usetikzlibrary{arrows}
\usepackage{scrextend}
\usepackage[makeroom]{cancel}
\usepackage{enumitem}
\usepackage{float}
\usepackage{graphicx} 

\definecolor{qqffff}{rgb}{0,1,1}
\definecolor{ttffcc}{rgb}{0.2,1,0.8}
\definecolor{ududff}{rgb}{0.30196078431372547,0.30196078431372547,1}
\definecolor{qqqqff}{rgb}{0,0,1}
\definecolor{ffffff}{rgb}{1,1,1}
\definecolor{ffqqqq}{rgb}{1,0,0}
\definecolor{qqffqq}{rgb}{0,1,0}
\definecolor{cqcqcq}{rgb}{0.7529411764705882,0.7529411764705882,0.7529411764705882}
\definecolor{xfqqff}{rgb}{0.4980392156862745,0,1}
\definecolor{ccwwff}{rgb}{0.8,0.4,1}
\definecolor{zzttqq}{rgb}{0.6,0.2,0}
\definecolor{cczzff}{rgb}{0.8,0.4,1}

\theoremstyle{definition}

\theoremstyle{plain}
\newtheorem{theorem}{Theorem}[section]
\newtheorem{lemma}[theorem]{Lemma}
\newtheorem{corollary}[theorem]{Corollary}
\newtheorem{proposition}[theorem]{Proposition}
\newtheorem{const}[theorem]{Construction}

\theoremstyle{definition}
\newtheorem{example}[theorem]{Example}

\newtheorem{notation}[theorem]{Notation}
\newtheorem{definition}[theorem]{Definition}

\newtheorem{remark}[theorem]{Remark}

\newcommand{\del}{\mbox{del}}
\newcommand{\lk}{\mbox{lk}}

\newcommand{\dist}{\mbox{dist}}

\newcommand{\hh}{\mbox{height}}

\newcommand{\lcm}{\mbox{lcm}}
\newcommand{\kk}{\Bbbk}

\newcommand\gen[1]{\left\langle #1 \right\rangle}

\newcommand\Set[2]{\left\lbrace #1 \mid #2 \right\rbrace}
\newcommand\set[1]{\left\lbrace #1 \right\rbrace}

\newcommand{\calN}{\mathcal{N}}
\newcommand{\calS}{\mathcal{S}}
\newcommand{\calI}{\mathcal{I}}
\newcommand{\calF}{\mathcal{F}}
\newcommand{\scrF}{\mathscr{F}}
\newcommand{\calA}{\mathcal{A}}

\newcommand{\MVC}{\raisebox{-0.2ex}{\scalebox{0.55}[1]{$\mathbf{MVC}$}}}
\newcommand{\VC}{\raisebox{-0.2ex}{\scalebox{0.55}[1]{$\mathbf{VC}$}}}
\newcommand{\MOTD}{\raisebox{-0.2ex}{\scalebox{0.55}[1]{$\mathbf{MOTD}$}}}
\newcommand{\MTD}{\raisebox{-0.2ex}{\scalebox{0.55}[1]{$\mathbf{MTD}$}}}
\newcommand{\TD}{\raisebox{-0.2ex}{\scalebox{0.65}[1]{$\mathbf{TD}$}}}
\newcommand{\OTD}{\raisebox{-0.2ex}{\scalebox{0.6}[1]{$\mathbf{OTD}$}}}
\allowdisplaybreaks

\title{Geometrically vertex decomposable open neighborhood ideals}
\author{Jounglag Lim}
\date{}

\begin{document}
\begin{abstract}
    In this paper, we prove that the open neighborhood ideal of a TD-unmixed tree is geometrically vertex decomposable. 
    This result implies that the associated Stanley-Reisner complex is vertex decomposable. 
    We further demonstrate that Cohen-Macaulay open neighborhood ideals of trees are special cases of Cohen-Macaulay facet ideals of simplicial trees. 
    Finally, we investigate open neighborhood ideals of chordal graphs and establish that almost all square-free monomial ideal can be realized as the open neighborhood ideal of a chordal graph.
\end{abstract}

\maketitle


\section{Introduction}

Since Stanley's proof of the Upper Bound Conjecture for simplicial spheres~\cite{Stan_UBC}, combinatorial commutative algebra has witnessed significant progress in connecting commutative algebra with the combinatorics of simplicial (multi)complexes.
In 1990, Villarreal established a fundamental link between graph theory and combinatorial commutative algebra.
Given a finite simple graph $G = (V,E)$, he defined the \emph{edge ideal} of $G$, denoted $\calI(G)$, to be the square-free monomial ideal in $R = \kk[V]$ generated by the edges of $G$.
Notably, he characterized the trees whose edge ideals are Cohen-Macaulay~\cite{MR1031197}.
Following Villarreal's work, numerous types of ideals associated to graphs have been introduced~\cite{MR1666661,MR4132629,MR2252121,MR4132629}.
In particular, Sharifan and Moradi defined the \emph{closed neighborhood ideal} of a graph and studied the Castelnuovo--Mumford regularity and projective dimension of such ideals~\cite{MR4132629}.
Subsequently, Honneycutt and Sather-Wagstaff characterized the trees whose closed neighborhood ideals are Cohen-Macaulay~\cite{MR4445927}.

In this paper, we consider the \emph{open neighborhood ideals} of trees, which are distinct from closed neighborhood ideals.
For a graph $G$, the \emph{open neighborhood ideal} of $G$, denoted $\calN(G)$, is the square-free monomial ideal generated by the open neighborhoods of the vertices of $G$.
The simplicial complex whose Stanley--Reisner ideal corresponds to $\calN(G)$ is called the \emph{stable complex}.
In prior work~\cite{COURAGE}, we characterized the trees whose open neighborhood ideals are Cohen--Macaulay (by proving that the associated stable complexes are shellable) and computed their Cohen--Macaulay type.

The main contributions of this paper are as follows:
\begin{itemize}
    \item[(1)] We show that Cohen--Macaulay open neighborhood ideals of trees are geometrically vertex decomposable (Theorem~\ref{thm. oni is GVD}), which implies that their associated Stanley--Reisner complexes are vertex decomposable.
    \item[(2)] We establish a relationship between the open neighborhood ideals of TD-unmixed trees and the Cohen--Macaulay facet ideals of simplicial trees (Theorem~\ref{thm. odd ONI is facet ideal}).
    \item[(3)] We determine which square-free monomial ideals can be realized as open neighborhood ideals of finite simple graphs (Corollary~\ref{cor. any SFM as ONI of chordal}).
\end{itemize}

In Section~\ref{sec. ONI and Scomp}, we provide background on open neighborhood ideals of trees and define the stable complex of a graph, which serves as the Stanley--Reisner complex of the open neighborhood ideal.
In Section~\ref{sec. geom. vert. decomp. and EScomp}, we define geometrically vertex decomposable square-free monomial ideals, a property equivalent to the vertex decomposability of the associated Stanley--Reisner complex.
We also introduce the odd-open neighborhood ideal of balanced trees, a modified version of open neighborhood ideal.
In Section~\ref{sec. ESC are GVD}, we prove that the odd-open neighborhood ideal of TD-unmixed balanced trees are geometrically vertex decomposable (Theorem~\ref{thm. oni is GVD}).
In Section~\ref{sec. ONI and facet ideals}, we address Item~(2).
Finally, in Section~\ref{sec. ONI of chordal}, we address Item~(3) by studying the open neighborhood ideals of chordal graphs.

\section*{Acknowledgments}
The author would like to thank Adam Van Tuyl for his numerous suggestions and feedback on the first draft of this paper.
\section{Preliminary}\label{sec. ONI and Scomp}

In this section, we review the preliminary definitions and results used throughout the paper.

\subsection{Vertex decomposable complexes and the Stanley-Reisner correspondence}

We begin with basic terminology regarding simplicial complexes and recall the Stanley--Reisner correspondence.
For a finite set $V$, a \emph{simplicial complex} on $V$ is a collection $\Delta \subseteq 2^V$ that is closed under inclusion; that is, if $F \in \Delta$ and $G \subseteq F$, then $G \in \Delta$.
The elements of $V$ are called the \emph{vertices} of $\Delta$, the elements of $\Delta$ are called \emph{faces}, and the maximal elements of $\Delta$ (with respect to inclusion) are called \emph{facets}.
We denote the set of facets of $\Delta$ by $\mathcal{F}(\Delta)$.
If $\mathcal{F}(\Delta) = \set{F_1,\dots,F_k}$, we write $\Delta = \gen{F_1,\dots,F_k}$. More generally, if $\mathcal{F}(\Delta) \subseteq S \subseteq \Delta$, we write $\Delta = \gen{S}$.
The \emph{dimension} of a face $F \in \Delta$ is defined as $\dim F := |F| - 1$, and the dimension of $\Delta$ is $\dim \Delta := \max \Set{\dim F}{F \in \Delta}$.
We say that $\Delta$ is \emph{pure} if every facet of $\Delta$ has the same dimension.

For a face $F \in \Delta$, the \emph{deletion} and the \emph{link} of $\Delta$ with respect to $F$ are the simplicial complexes defined by
$$
\del_\Delta(F) := \Set{G \in \Delta}{G \cap F = \emptyset}
$$
and
$$
\lk_\Delta(F) := \Set{G \in \Delta}{G \cap F = \emptyset, G \cup F \in \Delta}.
$$
A vertex $v \in V$ is a \emph{shedding vertex} of $\Delta$ if every facet of $\del_{\Delta}(v)$ is a facet of $\Delta$.
The following definition is due to Provan and Billera.

\begin{definition}[\protect{\cite[Definition 2.1]{joinIsShellable}}]\label{def. vert decomp}
    Let $\Delta$ be a pure simplicial complex.
    We say that $\Delta$ is \emph{vertex decomposable} if
    \begin{enumerate}
        \item $\Delta$ is a simplex or the empty set, or
        \item there exists a shedding vertex $v$ of $\Delta$ such that both $\del_\Delta(v)$ and $\lk_\Delta(v)$ are vertex decomposable.
    \end{enumerate}
\end{definition}

Let $\Delta'$ and $\Delta''$ be simplicial complexes on disjoint vertex sets $V'$ and $V''$, respectively.
The \emph{join} of $\Delta'$ and $\Delta''$ is the simplicial complex on $V' \cup V''$ defined by
$$
\Delta' \ast \Delta'' := \Set{F' \cup F''}{F' \in \Delta', F'' \in \Delta''}.
$$
The following result will be useful later in the paper.

\begin{theorem}[\protect{\cite[Proposition 2.4]{joinIsShellable}}]\label{thm. join is VD}
    Let $\Delta'$ and $\Delta''$ be pure simplicial complexes with disjoint vertex sets.
    Then $\Delta' \ast \Delta''$ is vertex decomposable if and only if both $\Delta'$ and $\Delta''$ are vertex decomposable.
\end{theorem}

Let $\kk[V]$ be a polynomial ring over a field $\kk$ whose variables are the elements of $V$.
The \emph{Stanley--Reisner ideal} of a simplicial complex $\Delta$ on $V$ is the square-free monomial ideal in $\kk[V]$ defined by
$$
I_\Delta := \gen{x_{i_1}\cdots x_{i_k} \ :\ \set{x_{i_1},\dots,x_{i_k}} \subseteq V, \set{x_{i_1},\dots,x_{i_k}} \notin \Delta},
$$
and the \emph{face ring} (or Stanley--Reisner ring) of $\Delta$ is the quotient ring $\kk[\Delta] := \kk[V]/I_{\Delta}$.
The primary decomposition of the Stanley--Reisner ideal is given by
$$
I_{\Delta} = \bigcap_{F \in \Delta} \gen{V \setminus F} = \bigcap_{F \in \mathcal{F}(\Delta)} \gen{V \setminus F},
$$
where the second decomposition is irredundant.
For a square-free monomial ideal $I \subseteq \kk[V]$, its \emph{Stanley--Reisner complex} is the simplicial complex on $V$ defined by
$$
\Delta_I := \Set{\set{x_{i_1},\dots,x_{i_k}} \subseteq V}{x_{i_1}\cdots x_{i_k} \notin I},
$$
with the convention that the empty set $\varnothing$ corresponds to the monomial $1$.
One can verify that for a simplicial complex $\Delta$ and a square-free monomial ideal $I$, we have $\Delta_{I_{\Delta}} = \Delta$ and $I_{\Delta_I} = I$.
This establishes a one-to-one correspondence between simplicial complexes on $V$ and square-free monomial ideals in $\kk[V]$, known as the \emph{Stanley--Reisner correspondence}.

\subsection{Open neighborhood ideals and Stable complexes}

In this subsection, we define open neighborhood ideals and stable complexes.
Throughout, every graph is assumed to be finite and simple.

\begin{notation}
    We write $\mathbb{N} := \set{1,2,3,\dots}$ and $\mathbb{N}_0 := \mathbb{N} \cup \set{0}$.
    For $n \in \mathbb{N}$, we define $[n] := \set{1,2,\dots,n}$ and $[n]_0 := \set{0,1,\dots,n}$.
    For a set $S$ and $k \in \mathbb{N}$, we use the notation $\binom{S}{k}$ to denote the set of all subsets of $S$ of size $k$.
    For a graph $G = (V,E)$, we identify the edges with subsets of $V$; hence $E \subseteq \binom{V}{2}$.
\end{notation}

\begin{definition}\label{def. ON, minimality}
    Let $G = (V,E)$ be a graph.
    For $v \in V$, the \emph{open neighborhood} of $v$ is the set
    $$
    N(v) = N_G(v) := \Set{u \in V}{\set{u,v} \in E},
    $$
    and for $S \subseteq V$, the \emph{open neighborhood} of $S$ is the set
    $$
    N(S) = N_G(S) := \bigcup_{v \in S} N(v).
    $$
    We say that a subset $S \subseteq V$ is \emph{minimal} (in $G$) if there exists no proper subset $S' \subsetneq S$ such that $N(S') = N(S)$.
\end{definition}

For convenience, we also define the \emph{closed neighborhood} of $v$ and $S$ as
$$
N[v] = N_G[v] := N_G(v) \cup \set{v} \quad \text{and} \quad N[S] = N_G[S] := N_G(S) \cup S,
$$
respectively.

\begin{definition}\label{def. ONI}
    Let $G = (V,E)$ be a graph, and let $R = \kk[V]$ be a polynomial ring over a field $\kk$ whose variables are the elements of $V$.
    For a subset $S \subseteq V$, let $X_S := \prod_{v \in S}v$ be the square-free monomial in $R$ with support $S$; we set $X_\varnothing = 1$ by convention.
    The \emph{open neighborhood ideal} of $G$ is the ideal
    $$
    \calN(G) := \gen{X_{N(v)}\ :\ v \in V} \subset R.
    $$
\end{definition}

\begin{example}\label{ex. 1}
    Let $T = (V,E)$ be the graph shown in Figure~\ref{fig: example 1}.

    \begin{figure}[ht]
        \centering
        \includegraphics[width=0.4\linewidth]{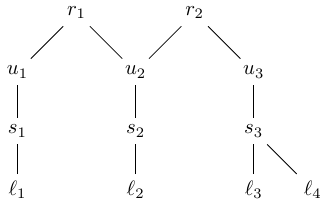}
        \caption{Graph $T$}
        \label{fig: example 1}
    \end{figure}

    \noindent We have $N(\ell_1) = \set{s_1}$, $N(s_1) = \set{\ell_1,u_1}$, $N(u_1) = \set{s_1,r_1}$, and $N(r_1) = \set{u_1,u_2}$.
    The open neighborhood ideal of $T$ is given by
    \begin{align*}
        \calN(T) &= \gen{s_1,s_2,s_3,\ell_1u_1,\ell_2u_2,\ell_3\ell_4u_3, s_1r_1, s_2r_1r_2, s_3r_3, u_1u_2,u_2,u_3}\\
        &= \gen{s_1,s_2,s_3,\ell_1u_1,\ell_2u_2,\ell_3\ell_4u_3,u_1u_2,u_2,u_3},
    \end{align*}
    where the second line is obtained by removing redundant generators.
\end{example}

Recall that a graph $G = (V,E)$ is a \emph{tree} if it is connected and contains no cycles, and $G$ is a \emph{forest} if every connected component of $G$ is a tree.
In the characterization of Cohen--Macaulay open neighborhood ideals of trees given in \cite{COURAGE}, a special class of trees played a central role, which we define next.

\begin{definition}\label{def. height, balanced trees}
    Let $G = (V,E)$ be a graph with at least one leaf.
    The \emph{height} of a vertex $v \in V$ is defined as
    $$
    \hh(v) = \hh_G(v) := \min\Set{\dist(v,\ell)}{\ell \mbox{ is a leaf in } G}.
    $$
    For vertices with $\deg_G(v) = 0$, we set $\hh(v) = 0$.
    We define the set of vertices at height $k$ by
    $$
    V_k = V_k(G) := \{v \in V\mid \hh(v) = k\},
    $$
    and the \emph{height} of $G$ is the integer
    $$
    \hh(G) := \max\Set{k \in \mathbb{N}}{V_k \neq \emptyset}.
    $$
    A tree (or forest) $T$ is called a \emph{balanced tree} (or \emph{forest}) if no two vertices of the same height are adjacent.
    Let $V_{odd} := \bigcup_{k = 0}^\infty V_{2k + 1}$ and $V_{even} := \bigcup_{k = 0}^\infty V_{2k}$.
    The \emph{odd-open neighborhood ideal} of a balanced tree (or forest) $T$ is the ideal
    $$
    \calN_{odd}(T) := \gen{X_{N(v)}\ :\ v \in V_{odd}}.
    $$
    Depending on the context, we view $\calN_{odd}(T)$ as an ideal in $\kk[V]$ or $\kk[V_{even}]$.
\end{definition}

\begin{example}\label{ex. 2}
    Consider the graph $T$ from Example~\ref{ex. 1}.
    Note that $T$ is a balanced tree of height 3.
    We have $V_0(T) = \set{\ell_1,\ell_2,\ell_3,\ell_4}$, $V_1(T) = \set{s_1,s_2,s_3}$, $V_2(T) = \set{u_1,u_2,u_3}$, and $V_3(T) = \set{r_1,r_2}$; hence $V_{odd}(T) = V_1 \cup V_3 = \set{s_1,s_2,s_3,r_1,r_2}$.
    The odd-open neighborhood ideal of $T$ is given by
    $$
    \calN_{odd}(T) = \gen{\ell_1u_1,\ell_2u_2,\ell_3\ell_4u_3,u_1u_2,u_2,u_3}.
    $$
\end{example}

Using the odd-open neighborhood ideals of balanced trees, we can express the open neighborhood ideal of a tree as the sum of two square-free monomial ideals sharing no variables among their minimal generators, as stated next.

\begin{theorem}[\protect{\cite[Theorem 5.2.1]{COURAGE}}]\label{theorem. sum of ONI}
    Let $T = (V,E)$ be a tree and let $R = \kk[V]$.
    Then there exist two subgraphs of $T$, denoted $T'$ and $T''$, such that:
    \begin{itemize}
        \item[(1)] $T'$ and $T''$ are balanced forests;
        \item[(2)] $V = V_{even}(T') \sqcup V_{even}(T'') \sqcup V_1(T)$, where $\sqcup$ denotes disjoint union;
        \item[(3)] $\calN(T) = \calN_{odd}(T') + \calN_{odd}(T'') + \gen{V_1(T)}$.
    \end{itemize}
\end{theorem}

\begin{example}\label{ex.3}
    Let $T = (V,E)$ be the tree shown in Figure~\ref{fig: ex 3}.
    Since $\hh_T(u_2) = 2 = \hh_T(u_3)$ and the vertices $u_2$ and $u_3$ are adjacent, $T$ is not a balanced tree.
    The subgraphs $T'$ and $T''$ of $T$ satisfying the conditions of Theorem~\ref{theorem. sum of ONI} are depicted in Figure~\ref{fig: ex 3ab}.
    \begin{figure}[h!]
        \centering
        \includegraphics[width=0.55\linewidth]{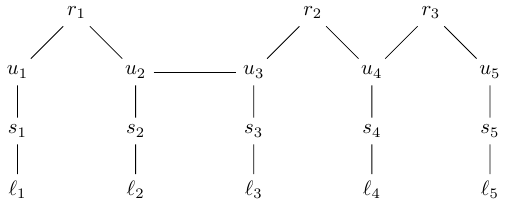}
        \caption{Graph $T$}
        \label{fig: ex 3}
    \end{figure}
    \begin{figure}[ht]
        \centering
        \includegraphics[width=0.344\linewidth]{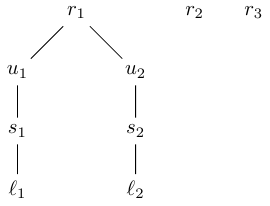}
        \qquad
        \vline
        \qquad
        \includegraphics[width=0.42\linewidth]{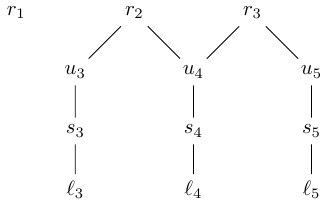}
        \caption{Subgraphs $T'$ (left) and $T''$ (right) of $T$}
        \label{fig: ex 3ab}
    \end{figure}
    Note that the isolated vertices $r_2$ and $r_3$ in $T'$ and $r_1$ in $T''$ have height 0. Details on the construction of $T'$ and $T''$ can be found in \cite[Definition 3.4.1]{COURAGE}.
\end{example}

\subsection{Characterization of TD-unmixed trees}

Analogous to the relationship between edge ideals and vertex covers, open neighborhood ideals are associated with total dominating sets.

\begin{definition}\label{def. TD-set, odd-TD-set}
    Let $G = (V,E)$ be a graph.
    A set $S \subseteq V$ is a \emph{total dominating set} (TD-set) if $N(S) = V$.
    Now let $T = (V,E)$ be a balanced tree.
    A set $S\subseteq V$ is an \emph{odd total dominating set} (odd-TD-set) of $T$ if $N(S) = V_{odd}$.
    We say that $G$ is (odd-) \emph{TD-unmixed} if every minimal (odd-) TD-set of $G$ has the same size.
\end{definition}

\begin{theorem}[\protect{\cite[Theorem 2.2.2]{COURAGE}}]\label{theorem. prime decomp. of ONI}
    Let $G = (V,E)$ be a graph, let $\TD(G)$ be the set of all TD-sets in $G$, and let $\MTD(G)$ be the set of all minimal TD-sets in $G$.
    Then the primary decomposition of $\calN(G)$ is given by
    $$
    \calN(G) = \bigcap_{S \in \TD(G)} \gen{S} = \bigcap_{S \in \MTD(G)}\gen{S},
    $$
    where the second decomposition is irredundant.
    Suppose that $G$ is a balanced tree.
    Let $\OTD(G)$ and $\MOTD(G)$ be the sets of all odd-TD-sets and minimal odd-TD-sets in $G$, respectively.
    Then
    $$
    \calN_{odd}(G) = \bigcap_{S \in \OTD(G)} \gen{S} = \bigcap_{S \in \MOTD(G)}\gen{S},
    $$
    where the second decomposition is irredundant.
\end{theorem}

\begin{example}\label{ex.4}
    Consider the trees $T$, $T'$, and $T''$ from Example~\ref{ex.3}.
    The sets of minimal odd-TD-sets of $T'$ and $T''$ are given by
    \begin{align*}
        \MOTD(T') &= \set{\set{u_1,u_2}, \set{\ell_1,u_2}, \set{u_1,\ell_2}},\\
        \MOTD(T'') &= \set{\set{u_3,u_4,u_5}, \set{\ell_3,u_4,u_5}, \set{u_3,\ell_4,u_5}, \set{u_3,u_4,\ell_5}, \set{\ell_3,u_4,\ell_5}}.
    \end{align*}
    By Theorem~\ref{theorem. prime decomp. of ONI}, we obtain
    \begin{align*}
        \calN_{odd}(T') &= \gen{u_1,u_2}\cap \gen{\ell_1,u_2}\cap \gen{u_1,\ell_2},\\
        \calN_{odd}(T'') &= \gen{u_3,u_4,u_5}\cap \gen{\ell_3,u_4,u_5}\cap\gen{u_3,\ell_4,u_5} \cap \gen{u_3,u_4,\ell_5} \cap \gen{\ell_3,u_4,\ell_5}.
    \end{align*}
    Furthermore, applying Theorem~\ref{theorem. sum of ONI}(3), we have
    \begin{align*}
        \calN(T) &= \bigcap_{S' \in \MOTD(T')}\left( \bigcap_{S'' \in \MOTD(T'')} \left(\gen{S'} + \gen{S''} + \gen{V_1(T)}\right)\right)\\
        &= \bigcap_{S' \in \MOTD(T')}\left( \bigcap_{S'' \in \MOTD(T'')} \left(\gen{S'} + \gen{S''} + \gen{s_1,s_2,s_3,s_4,s_5}\right)\right).
    \end{align*}
    In other words, every minimal TD-set of $T$ can be expressed as a disjoint union $S' \cup S'' \cup V_1(T)$, where $S' \in \MOTD(T')$ and $S'' \in \MOTD(T'')$.
\end{example}

\begin{definition}\label{def. even ruled complex of delta trees}
    Let $G = (V,E)$ be a graph.
    The \emph{stable complex} of $G$ is the simplicial complex
    $$
    \calS(G) := \Set{V \setminus D}{D\mbox{ is a TD-set in }G} = \gen{V \setminus D \mid D\mbox{ is a minimal TD-set in }G}.
    $$
    Suppose that $G$ is a balanced tree.
    The \emph{even-stable complex} of $G$ is the simplicial complex on $V_{even}$ defined by
    $$
    \calS_{even}(G) := \gen{V_{even}\setminus D \mid D \mbox{ is a minimal odd-TD-set of }G}.
    $$
\end{definition}

One can verify that $I_{\calS(G)} = \calN(G)$ and $I_{\calS_{even}(G)} = \calN_{odd}(G)$ by comparing the primary decomposition of Stanley--Reisner ideals with Theorem~\ref{theorem. prime decomp. of ONI}.
Using Theorem~\ref{theorem. sum of ONI} and Definition~\ref{def. even ruled complex of delta trees}, we obtain the following corollary.

\begin{corollary}[\protect{\cite[Theorem 4.1.9]{COURAGE}}]\label{cor. SC is join of two even SC}
    Let $T = (V,E)$ be a tree, and let $T'$ and $T''$ be subgraphs of $T$ satisfying the conditions in Theorem~\ref{theorem. sum of ONI}.
    Then
    $$
    \calS(T) = \calS_{even}(T') \ast \calS_{even}(T'').
    $$
\end{corollary}

\begin{example}\label{ex. 5}
    Let $T$, $T'$, and $T''$ be the trees from Example~\ref{ex.3}.
    From Example~\ref{ex.4}, we can compute the facets of the even-stable complexes of $T'$ and $T''$:
    \begin{align*}
        \calS_{even}(T') &= \gen{\set{r_2,r_3,\ell_1,\ell_2},\set{r_2,r_3,u_1,\ell_2},\set{r_2,r_3,\ell_1,u_2}},\\
        \calS_{even}(T'') &= \gen{\set{r_1,\ell_3,\ell_4,\ell_5},\set{r_1,u_3,\ell_4,\ell_5},\set{r_1,\ell_3,u_4,\ell_5},\set{r_1,\ell_3,\ell_4,u_5},\set{r_1,u_3,\ell_4,u_5}}.
    \end{align*}
    Notice that the vertices $r_1$, $r_2$, and $r_3$ appear in every facet because $\hh_{T'}(r_2) = \hh_{T'}(r_3) = 0$ and $\hh_{T''}(r_1) = 0$ are even.
    The observation regarding the minimal TD-sets of $T$ at the end of Example~\ref{ex.4} implies that the stable complex of $T$ is the join of the even-stable complexes of $T'$ and $T''$.
\end{example}

We note that the stable complex of a tree $T$ (respectively, the even-stable complex of a balanced tree $T$) is pure if and only if $T$ is TD-unmixed (respectively, odd-TD-unmixed).
Therefore, by Theorem~\ref{thm. join is VD} and Corollary~\ref{cor. SC is join of two even SC}, to show that every stable complex of a TD-unmixed tree is vertex decomposable, it suffices to show that every even-stable complex of an odd-TD-unmixed balanced tree is vertex decomposable.
Moreover, we only need to consider TD-unmixed balanced trees, as justified by the following theorem.

\begin{theorem}[\protect{\cite[Corollary 3.4.8]{COURAGE}}]\label{thm. char of unmixed trees}
    Let $T$ be a tree, and let $T'$ and $T''$ be subgraphs of $T$ satisfying the conditions in Theorem~\ref{theorem. sum of ONI}.
    Then $T$ is TD-unmixed if and only if $T'$ and $T''$ are TD-unmixed.
\end{theorem}

It implies that a balanced tree is odd-TD-unmixed if and only if it is TD-unmixed (see \cite[Section 3]{COURAGE}).
The following theorem gives a descriptive characterization of TD-unmixed balanced trees.

\begin{theorem}[\protect{\cite[Theorem 3.3.11]{COURAGE}}]
\label{theorem. char. of unmixed delt. trees}
    Let $T$ be a balanced tree.
    Then $T$ is TD-unmixed if and only if
    \begin{enumerate}
        \item $\hh(T) \leq 3$;
        \item for all $v \in V_2$, $|N(v) \cap V_1| = 1$; and
        \item for all $v \in V_1$, $|N(v) \cap V_2| \leq 1$.
    \end{enumerate}
    In particular, Condition \textit{(3)} becomes an equality if $\hh(T) = 3$.
\end{theorem}

\begin{example}\label{ex. unmixed delta trees}
    By Condition~\textit{(2)} of Theorem~\ref{theorem. char. of unmixed delt. trees}, no TD-unmixed balanced tree of height 2 exists.
    Hence, a TD-unmixed balanced tree is either an isolated vertex, a \emph{star graph} with at least two leaves (as shown in Figure~\ref{fig: fig 5}), or a TD-unmixed balanced tree of height 3.
    \begin{figure}[ht]
        \centering
        \includegraphics[width=0.25\linewidth]{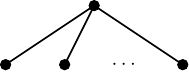}
        \caption{Star graph with at least 2 leaves}
        \label{fig: fig 5}
    \end{figure}
\end{example}

A constructive characterization of TD-unmixed balanced trees of height 3 is given in Theorem~\ref{thm. constructing hh = 3 delt. trees}.

\begin{notation}\label{not. path graph}
    For every $n \in \mathbb{N}_0$, let $P_n$ denote the path graph with vertex set $V(P_n) = [n]_0$ and edge set $E(P_n) = \set{\set{0,1},\set{1,2},\dots,\set{n-1,n}}$.
    Given a graph $G = (V,E)$ and a positive integer $k$, a sequence of distinct vertices $v_0,v_1,\dots,v_k \in V$ is called a \emph{path of length $k$} in $G$ if $\set{v_i,v_{i+1}} \in E$ for all $i \in [k - 1]_0$. We denote such a path by $v_0v_1\cdots v_k$.
\end{notation}

We first define an operation on balanced trees.

\begin{definition}
    Let $G_1 = (V',E')$ and $G_2 = (V'',E'')$ be graphs with disjoint vertex sets, and let $e = \set{v',v''}$ where $v' \in V'$ and $v'' \in V''$.
    The \emph{edge join} of $G_1$ and $G_2$ along $e$ is the graph
    $$
    G_1 +_e G_2 := (V'\cup V'', E' \cup E'' \cup \set{\set{v',v''}}).
    $$
    Let $T = (V,E)$ be a balanced tree of height 3, and let $v \in V\setminus V_0$.
    We define the function $\mathcal{O}$ as follows:
    \begin{enumerate}
        \item If $\hh(v) = 1$, then $\mathcal{O}(T,v) := T +_{\set{v,0}} P_0$.
        \item If $\hh(v) = 2$, then $\mathcal{O}(T,v) := T +_{\set{v,3}} P_3$.
        \item If $\hh(v) = 3$, then $\mathcal{O}(T,v) := T +_{\set{v,2}} P_2.$
    \end{enumerate}
    For $v_1 \in V(T)$ and $v_2 \in V(\mathcal{O}(T,v_1))$, we denote
    $$
    \mathcal{O}(T,(v_1,v_2)) := \mathcal{O}(\mathcal{O}(T,v_1),v_2).
    $$
    Inductively, for $n \geq 3$, if $v_i \in V(\mathcal{O}(T,(v_1,\dots,v_{i-1})))$ for all $2 \leq i \leq n$, we define
    $$
    \mathcal{O}(T,(v_1,\dots,v_n)) := \mathcal{O}(\mathcal{O}(T,(v_1,\dots,v_{n-1})),v_n).
    $$
\end{definition}

\begin{remark}
    When applying $\mathcal{O}$ to a graph $T$ with $v \in V(T)$, we must relabel the vertices of $T' := \mathcal{O}(T,v)$ so that $V(T') \cap \mathbb{N}_0 = \emptyset$. Since the vertices of $P_k$ are denoted by natural numbers, this relabeling ensures there are no two vertices with the same label.
    We assume that this procedure is performed automatically whenever $\mathcal{O}$ is applied.
\end{remark}

\begin{theorem}[\protect{\cite{COURAGE}}]\label{thm. constructing hh = 3 delt. trees}
    Let $T$ be an unmixed balanced tree of height 3.
    Then either $T \cong P_6$, or there exists a sequence of vertices $v_1,\dots,v_k$ such that $T \cong \mathcal{O}(P_6,(v_1,\dots,v_k))$.
\end{theorem}

\begin{example}\label{ex. 6}
    The balanced trees shown in Figure~\ref{fig: 6} are obtained by applying $\mathcal{O}$ to $P_6$ using vertices of heights 1, 2, and 3, respectively. The newly added vertices and edges are colored red.

    \begin{figure}[ht]
        \centering
        \includegraphics[width=1\linewidth]{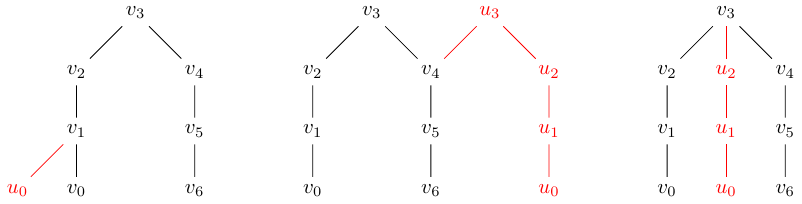}
        \caption{Left to right: $\mathcal{O}(P_6,v_1)$, $\mathcal{O}(P_6,v_4)$, and $\mathcal{O}(P_6,v_3)$}
        \label{fig: 6}
    \end{figure}
\end{example}
\section{Geometrically vertex decomposability and even-stable complexes}\label{sec. geom. vert. decomp. and EScomp}

In this section, we define geometrically vertex decomposable ideals, which is the main tool we use to show that the stable complexes of TD-unmixed trees are vertex decomposable.

\begin{definition}[\protect{\cite[Definition 2.3, 2.6]{GVD_KR}}]\label{def. GVD}
    Let $I \subset R := \kk[x_1,\dots,x_n]$ be a square-free monomial ideal generated by the monomials $m_1,\dots,m_k$.
    For an indeterminate $y = x_i$, define
    $$
    N_{y,I} := \gen{m_j\ :\ y \nmid m_j} \quad \text{and} \quad C_{y,I} := \gen{\frac{m_j}{y}\ :\ y \mid m_j} + N_{y,I}.
    $$
    We say $I$ is \emph{geometrically vertex decomposable} (GVD) if $I$ is unmixed (i.e., every minimal prime of $I$ has the same height), and if
    \begin{itemize}
        \item[(1)] $I = \gen{1}$ or $I$ is generated by indeterminates in $R$, or
        \item[(2)] there exists an indeterminate $y$ in $R$ such that $I = C_{y,I} \cap (N_{y,I} + \gen{y})$, and the ideals $C_{y,I}$ and $N_{y,I}$ are geometrically vertex decomposable in $R/\gen{y}$.
    \end{itemize}
\end{definition}

\begin{remark}
    The decomposition in (2) is called a \emph{geometric vertex decomposition} of $I$ with respect to $y$, first introduced in \cite{GVD_KMY}.
    GVD is defined for a more general class of ideals using Gr\"{o}bner geometry for liaison theory.
    However, for our purpose, we only define the square-free monomial ideal version of GVD.
\end{remark}

\begin{theorem}[\protect{\cite[Proposition 2.9]{GVD_KR}}]\label{theorem. GVD iff vert decomp}
    A pure simplicial complex $\Delta$ is vertex decomposable if and only if its Stanley--Reisner ideal $I_\Delta$ is geometrically vertex decomposable.
\end{theorem}

Since the simplicial complex of our interest is defined by its minimal non-faces (i.e., the generating set of its Stanley--Reisner ideal), it is natural to use the GVD property to determine whether the simplicial complex is vertex decomposable.
We list some basic properties of GVD square-free monomial ideals below.

\begin{lemma}\label{lem. prod of var is GVD}
    Let $I = \gen{m} \subseteq \kk[x_1,\dots,x_n]$ for some square-free monomial $m \in \kk[x_1,\dots,x_n]$.
    Then $I$ is GVD.
\end{lemma}

\begin{proof}
    Without loss of generality, let $m = x_1\cdots x_k$ for some integer $k$ with $1 \leq k \leq n$.
    We proceed by induction on $k$.
    If $k = 1$, then $I$ is generated by an indeterminate, so it is GVD by definition.
    If $k > 1$, consider the decomposition with respect to $x_k$:
    \begin{align*}
        C_{x_k, I} \cap (N_{x_k, I} + \gen{x_k}) &= \gen{x_1\cdots x_{k-1}} \cap (\gen{0} + \gen{x_k})\\
        &= \gen{x_1\cdots x_{k-1}} \cap \gen{x_k}\\
        &= \gen{x_1\cdots x_k} = I.
    \end{align*}
    Since the zero ideal $\gen{0}$ is GVD and $\gen{x_1\cdots x_{k-1}}$ is GVD by the induction hypothesis, we conclude that $I$ is GVD.
\end{proof}

Using Theorem~\ref{theorem. GVD iff vert decomp} and Theorem~\ref{thm. join is VD}, we obtain the following result.
Note that a more general version of this theorem was established by Cummings, Da Silva, Rajchgot, and Van Tuyl~\cite{GVD_toric}.

\begin{theorem}[\protect{\cite[Theorem 1.1]{GVD_toric}}]\label{theorem. I + J is GVD}
    Let $I, J \subseteq \kk[x_1,\dots,x_n]$ be square-free monomial ideals whose minimal generators use disjoint sets of variables.
    Then $I + J$ is GVD if and only if both $I$ and $J$ are GVD.
\end{theorem}

\begin{corollary}\label{cor. GVD SMI plus a var is GVD}
    Let $m_1,\dots,m_k \in \kk[x_1,\dots,x_n,y]$ be square-free monomials such that $y \nmid m_i$ for all $1 \leq i \leq k$.
    If $I = \gen{m_1,\dots,m_k}$ is GVD, then $J = I + \gen{y}$ is also GVD.
\end{corollary}

The remainder of this section discusses some graph-theoretic properties of balanced trees which are used in Section~\ref{sec. ESC are GVD}.
We also define a special type of open neighborhood ideal at the end of this section.

\begin{theorem}[\protect{\cite[Corollary 3.5.5]{COURAGE}}]\label{theorem. height 2 vertex of degree 2 exists}
    Let $T = (V,E)$ be a TD-unmixed balanced tree of height 3.
    Then there exists a vertex $u \in V_2$ such that $\deg(u) = 2$.
\end{theorem}

Let $G = (V,E)$ be a graph.
For $S \subset V$, the \emph{subgraph induced by $S$}, denoted $G[S]$, is the subgraph of $G$ whose vertex set is $S$ and where $\set{u,v} \subset S$ is an edge in $G[S]$ if and only if $\set{u,v} \in E$.
We denote the subgraph induced by $V \setminus S$ by $G \setminus S$.

The following theorem is used in the proof of Proposition~\ref{prop. T - r is delta forest} and can be proved using the equivalent definitions of a tree (see \cite[Theorem 1.5.1]{GT_Diestel}).

\begin{theorem}\label{theorem. connected components of tree}
    Let $T = (V,E)$ be a tree and let $v \in V$.
    Then $T \setminus v$ has exactly $\deg(v)$ connected components, and each connected component contains a unique vertex adjacent to $v$ in $T$.
\end{theorem}

\begin{proposition}\label{prop. T - r is delta forest}
    Let $T$ be a TD-unmixed balanced tree of height 3, and let $r \in V_3$.
    Then $T' := T \setminus \set{r}$ is a TD-unmixed balanced forest.
\end{proposition}

\begin{proof}
    Since we only removed edges between height 2 and height 3 vertices when we delete $r$ from $T$, conditions (1), (2), and (3) in Theorem~\ref{theorem. char. of unmixed delt. trees} still holds in $T \setminus r$.
    Hence it suffices to show that every connected component of $T'$ is a balanced tree.
    Let $C$ be a connected component in $T'$.
    Then by Theorem~\ref{theorem. connected components of tree}, there is $u \in V(C)$ such that $u \in N_T(r)$.
    Since $\hh(r) = 3$, we must have $\hh(u) = 2$ as $T$ is balanced.
    
    If $\deg_T(u) = 2$, then $u$ is a leaf in $C$.
    Also by Theorem~\ref{theorem. char. of unmixed delt. trees}, there is $s \in V_1(T)$ such that $N_C(u) = \set{s}$.
    In this case $C$ is the subgraph induced by $N[s]$ which is a star graph, an unmixed balanced tree of height 1.

    Now suppose that $\deg_T(u) > 2$.
    Then by Theorem~\ref{theorem. char. of unmixed delt. trees}, there is at least one vertex $r' \in V_3(T)$ that is in $C$.
    We claim that $C$ is a balanced tree of height 3.
    First, note that for all $x \in V(C)\setminus \set{u'}$, we have $\deg_C(x) = \deg_T(x)$ and $\deg_C(u') \geq 2$. 
    Hence every leaf in $C$ is also a leaf in $T$.
    
    Let $v \in V(C)$, and let $h:= \hh_T(v)$.
    If $h = 0$, then $v$ is a leaf in $T$.
    Since there are no edge added or deleted among the edges incident to height 0 or 1 vertices in $T$, $v$ is still a leaf in $C$.
    Thus $\hh_C(v) = 0 = \hh_T(v)$.
    Hence we have shown that: for all $x \in V(C)$, $x$ is a leaf in $C$ if and only if $x$ is a leaf in $T$ ($\star$).
    
    If $h = 1$, then there is a leaf $v_0 \in V(T)$ that is adjacent to $v$.
    Since $\hh_T(r) = 3$, $r \neq v$ and $r \neq v_0$.
    So, $\set{v,v_0} \in E(C)$, hence $v_0 \in C$, therefore $\hh_C(v) = 1 = \hh_T(v)$ by ($\star$).

    If $h = 2$, then there are vertices $v_0 \in V_0(T)$ and $v_1 \in V_1(T)$ such that $v_0v_1v$ is a path of length 2 in $T$.
    Since the heights of these vertices are different from $\hh_T(r)$, $r \not\in \set{v_0,v_1,v}$.
    Thus $\set{v_0,v_1,v} \subseteq V(T')$.
    Also, as $v \in V(C)$ and $\set{v_0v_1, v_1v} \subset E(T')$, $\set{v_0,v_1,v} \subset V(C)$.
    So, we get $\hh_C(v) = 2 = \hh_T(v)$ by ($\star$).
    
    Finally, suppose that $h = 3$.
    Then there are vertices $v_i \in V_i(T)$ for $i \in [2]_0$ such that $v_0v_1v_2v$ is a path of length 3 in $T$.
    Since $v \in V(C) \subset V(T')$ but $r \not\in V(T')$, $r \neq v$.
    Thus $r \not\in \set{v_0,v_1,v_2,v}$ by the height argument as before.
    Hence $\set{v_0v_1, v_1v_2, v_2v} \subset E(T')$ and $v \in V(C)$, we get $\set{v_0,v_1,v_2} \subset V(C)$.
    Therefore, $\hh_C(v) = 3 = \hh_T(v)$ by ($\star$).

    So far, we have shown that $\hh_C(x) = \hh_T(x)$ for all $x \in V(C)$.
    Since $E(C) \subset E(T)$, $T$ being a balanced tree implies that $C$ is a balanced tree as desired.
\end{proof}

A direct consequence of the proof of Proposition~\ref{prop. T - r is delta forest} is that the set of vertices of odd height in $T \setminus \set{r}$ is exactly the set of vertices of odd height in $T$ excluding $r$.

\begin{corollary}\label{cor. odd in T - r is odd in T}
    Let $T$ be a TD-unmixed balanced tree of height 3, and let $r \in V_3$.
    Then $V_{odd}(T \setminus \set{r}) = V_{odd}(T) \setminus \set{r}$.
\end{corollary}

\begin{example}\label{ex. 7}
    Consider the TD-unmixed balanced tree $T$ shown in Figure~\ref{fig: 7}.
    \begin{figure}[ht]
        \centering
        \includegraphics[width=0.3\textwidth]{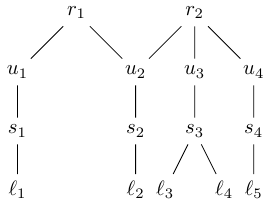}
        \quad
        \vline
        \quad
        \includegraphics[width=0.235\textwidth]{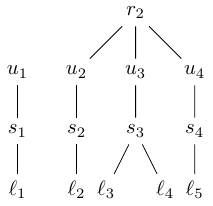}
        \quad
        \vline
        \quad
        \includegraphics[width=0.3\textwidth]{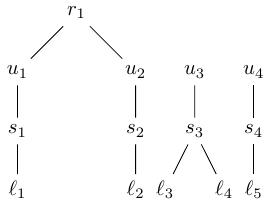}
        \caption{$T$, $T\setminus r_1$, and $T \setminus r_2$ (left to right, respectively)}
        \label{fig: 7}
    \end{figure}
    
    \noindent The graphs $T\setminus r_1$ and $T \setminus r_2$ are TD-unmixed balanced forests, as guaranteed by Proposition~\ref{prop. T - r is delta forest}.
    Furthermore, we have $V_{odd}(T\setminus r_1) = \set{s_1,s_2,s_3,s_4,r_2} = V_{odd}(T)\setminus \set{r_1}$ and $V_{odd}(T\setminus r_2) = \set{s_1,s_2,s_3,s_4,r_1} = V_{odd}(T)\setminus \set{r_2}$, consistent with Corollary~\ref{cor. odd in T - r is odd in T}.
\end{example}

\begin{theorem}\label{thm. T - N[u] is delta forest}
    Let $T = (V,E)$ be a TD-unmixed balanced tree of height 3, and let $u \in V_2$.
    Then $T\setminus N[u]$ is a TD-unmixed balanced forest.
\end{theorem}

\begin{proof}
    By Theorem~\ref{theorem. char. of unmixed delt. trees}(2), we can write $N_T(u) = \set{r_1,\dots,r_k,s}$ for some $k \geq 1$, where $\hh_T(r_i) = 3$ for all $1 \leq i \leq k$ and $\hh_T(s) = 1$.
    Since $T$ is a tree, each connected component of $T\setminus N[u]$ contains a unique vertex adjacent to a vertex in $N_T(u)$.
    Let $C$ be a connected component of $T\setminus N[u]$.
    If a vertex in $C$ is adjacent to $r_i$ for some $i \in [k]$ in $T$, then $C$ is also one of the connected components of $T \setminus \set{r_i}$.
    Hence, by Proposition~\ref{prop. T - r is delta forest}, $C$ is a TD-unmixed balanced tree.
    If a vertex in $C$ is adjacent to $s$, then $C$ must be an isolated vertex by Theorem~\ref{theorem. char. of unmixed delt. trees}(3), which is a TD-unmixed balanced tree of height 0.
\end{proof}

\begin{corollary}\label{cor. odd in T-N[u] is odd in T}
    Let $T = (V,E)$ be a TD-unmixed balanced tree of height 3, and let $u \in V_2$.
    Then $V_{odd}(T\setminus N[u]) = V_{odd}(T)\setminus N_T(u)$.
\end{corollary}

\begin{example}\label{ex. 8}
    Consider the tree $T$ from Example~\ref{ex. 7}.
    The forests $T\setminus N[u_1]$, $T \setminus N[u_2]$, and $T \setminus N[u_3]$ are shown in Figure~\ref{fig: 8}.
    \begin{figure}[ht]
        \centering
        \includegraphics[width=0.28\linewidth]{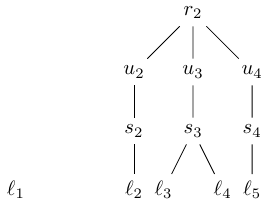}
        \quad
        \vline
        \quad
        \includegraphics[width=0.28\linewidth]{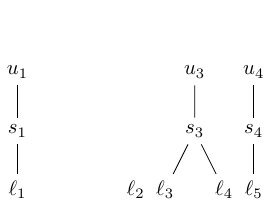}
        \quad
        \vline
        \quad
        \includegraphics[width=0.28\linewidth]{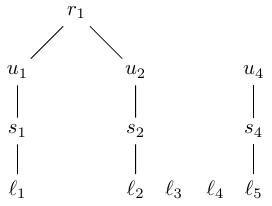}
        \caption{$T \setminus N[u_1]$, $T \setminus N[u_2]$, and $T \setminus N[u_3]$ (left to right, respectively)}
        \label{fig: 8}
    \end{figure}
    As stated in Theorem~\ref{thm. T - N[u] is delta forest}, all three graphs above are indeed TD-unmixed balanced forests.
    Furthermore, we have $V_{odd}(T\setminus N[u_i]) = V_{odd}(T)\setminus N_T(u_i)$ for all $i \in \{1,2,3\}$, consistent with Corollary~\ref{cor. odd in T-N[u] is odd in T}.
\end{example}

Here is the special type of open neighborhood ideal mentioned earlier.

\begin{definition}\label{def. induced oni}
    Let $T$ be a graph with at least one leaf, and let $T'$ be a subgraph of $T$.
    The \emph{induced odd-open neighborhood ideal of $T'$} in $T$ is the ideal in $\kk[V(T')]$ given by
    $$
    \calN_{odd}(T', T) := \gen{X_{N_{T'}(v)}\ :\ v \in V(T') \cap V_{odd}(T)}.
    $$
    If $T$ is a balanced tree, then we treat $\calN_{odd}(T',T)$ as an ideal in $\kk[V(T') \cap V_{even}(T)]$.
\end{definition}

\begin{example}\label{ex. 9}
    Let $T$ and $T'$ be the graphs shown in Figure~\ref{fig: 9}.
    Notice that $T$ is a balanced tree and $T' = T \setminus \set{u_3}$. However, $T'$ is not a balanced forest because $\hh_{T'}(u_1) = 2 = \hh_{T'}(r_1)$ and $\set{u_1,r_1} \in E(T')$.

    \begin{figure}[ht]
        \centering
        \includegraphics[width=0.3\linewidth]{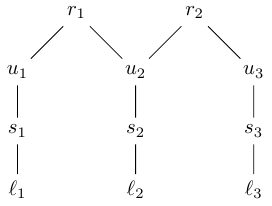}
        \qquad
        \vline
        \qquad
        \includegraphics[width=0.3\linewidth]{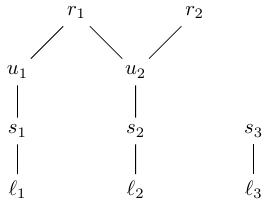}
        \caption{$T$ (left) and $T' = T \setminus u_3$ (right)}
        \label{fig: 9}
    \end{figure}

    \noindent We have $V_{odd}(T) = \set{s_1,s_2,s_3, u_1,u_2,u_3}$, while $V_{odd}(T') = \set{s_1,s_2,u_2}$.
    Thus, $\calN_{odd}(T) = \gen{\ell_1u_1, \ell_2u_2, \ell_3u_3, u_1u_2, u_2u_3}$ and $\calN_{odd}(T') = \gen{\ell_1u_1, \ell_2u_2, s_2r_1r_2}$, whereas
    $$
    \calN_{odd}(T',T) = \gen{\ell_1u_1, \ell_2u_2, \ell_3, u_1u_2, u_2} = \gen{\ell_1u_1, \ell_3, u_2}.
    $$
    Let $I = \calN_{odd}(T)$.
    Then we obtain $N_{u_3,I} = \gen{\ell_1u_1, \ell_2u_2, u_1u_2}$ and
    \begin{align*}
        C_{u_3,I} &= \gen{\ell_3,u_2} + N_{u_3,I}\\
        &= \gen{\ell_3,u_2,\ell_1u_1, \ell_2u_2, u_1u_2}\\
        &= \gen{\ell_1u_1, \ell_3, u_2}\\
        &= \calN_{odd}(T',T).
    \end{align*}
    The equality $C_{u_3, I} = \calN_{odd}(T',T)$ holds in general, as stated later in Lemma~\ref{lem. Cu,I and Nu,I}.
\end{example}
\section{$\calN(T)$ is Geometrically vertex decomposable}\label{sec. ESC are GVD}

In this section, we prove that open neighborhood ideals of TD-unmixed trees are geometrically vertex decomposable (GVD).
We begin with a lemma concerning TD-unmixed balanced trees of height 0 or 1.

\begin{lemma}\label{lem. height 0 and 1 are GVD}
    Let $T$ be a TD-unmixed balanced tree of height 0 or 1.
    Then $\calN_{odd}(T)$ is GVD.
\end{lemma}

\begin{proof}
    If $\hh(T) = 0$, then $\calN_{odd}(T) = \gen{0}$, which is GVD.
    If $\hh(T) = 1$, then $T$ consists of one vertex $s$ of height 1 and leaves $\ell_1,\dots,\ell_k$ with $k \geq 2$ (see Example~\ref{ex. unmixed delta trees}).
    Hence we have $\calN_{odd}(T) = \gen{\ell_1\cdots\ell_k}$, which is GVD by Lemma~\ref{lem. prod of var is GVD}.
\end{proof}

Next, we prove a lemma that characterizes $C_{u,I}$ and $N_{u,I}$ in terms of induced odd-open neighborhood ideals when $T$ is a TD-unmixed balanced tree of height 3, $I = \calN_{odd}(T)$, and $u$ is a vertex of height 2 with degree 2.

\begin{lemma}\label{lem. Cu,I and Nu,I}
    Let $T$ be a TD-unmixed balanced tree of height 3, and let $u \in V_2(T)$ with $\deg_T(u) = 2$.
    Set $I := \calN_{odd}(T)$, $T' := T \setminus \set{u}$, and $T'' := T\setminus N[u]$.
    Then $C_{u,I} = \calN_{odd}(T', T)$ and $N_{u,I} = \calN_{odd}(T'',T) = \calN_{odd}(T'')$.
\end{lemma}

\begin{proof}
    Since $\deg_T(u) = 2$, there exist vertices $r \in V_3(T)$ and $s \in V_1(T)$ such that $N_T(u) = \set{r,s}$.
    We can write
    $$
    I = \gen{X_{N(v)}\ :\ v \in V_{odd}(T)\setminus \set{r,s}} + \gen{X_{N(r)},X_{N(s)}}.
    $$
    The monomials $X_{N(r)}$ and $X_{N(s)}$ are the only minimal generators of $I$ divisible by $u$.
    Hence,
    \begin{align*}
        C_{u,I} = \gen{X_{N(v)}\ :\ v \in V_{odd}(T)\setminus \set{r,s}} + \gen{\frac{X_{N(r)}}{u},\frac{X_{N(s)}}{u}} = \calN_{odd}(T',T)
    \end{align*}
    by Definition~\ref{def. induced oni} and the construction of $T'$.
    Similarly, we have
    \begin{align*}
        N_{u,I} := \gen{X_{N(v)}\ :\ v \in V_{odd}(T)\setminus \set{r,s}} = \calN_{odd}(T'',T),
    \end{align*}
    since $V(T'') \cap V_{odd}(T) = V_{odd}(T) \setminus \set{r,s}$ by Corollary~\ref{cor. odd in T-N[u] is odd in T}, and the deletion of $N[u] = \set{u,r,s}$ from $T$ does not affect the open neighborhoods of vertices in $V_{odd}(T)\setminus \set{r,s}$.
    Again, by Corollary~\ref{cor. odd in T-N[u] is odd in T}, we conclude that $\calN_{odd}(T'',T) = \calN_{odd}(T'')$.
\end{proof}

\begin{example}\label{ex. 10}
    Consider the graphs $T$ and $T' = T\setminus \set{u_3}$ from Example~\ref{ex. 9}.
    Set $I = \calN_{odd}(T)$.
    Since $\deg_T(u_3) = 2$, we have $C_{u_3,I} = \calN_{odd}(T',T)$, as shown in Example~\ref{ex. 9}.
    Now let $T'' = T \setminus N[u_3]$ (see Figure~\ref{fig: 10}).
    \begin{figure}[ht]
        \centering
        \includegraphics[width=0.3\linewidth]{figures/fig_9a.pdf}
        \qquad
        \vline
        \qquad
        \includegraphics[width=0.3\linewidth]{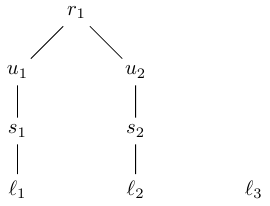}
        \caption{$T$ (left) and $T'' = T \setminus N[u_3]$ (right)}
        \label{fig: 10}
    \end{figure}
    From Example~\ref{ex. 9}, we obtained $N_{u_3,I} = \gen{\ell_1u_1,\ell_2u_2, u_1u_2}$.
\end{example}

By Theorem~\ref{thm. T - N[u] is delta forest}, $N_{u,I}$ is the odd-open neighborhood ideal of a TD-unmixed balanced forest.
However, the situation is more complicated for $C_{u,I}$ (associated with $T'$).
We will address this case in the proof of Theorem~\ref{thm. oni is GVD}, which constitutes the core of the argument.

\begin{theorem}\label{thm. oni is GVD}
    Let $T$ be a TD-unmixed balanced tree.
    Then $I := \calN_{odd}(T) \subseteq \kk[V(T)]$ is GVD.
\end{theorem}

\begin{proof}
    If $\hh(T) < 3$, then Lemma~\ref{lem. height 0 and 1 are GVD} covers the case.
    So suppose that $\hh(T) = 3$.
    Then there exists $u \in V_2(T)$ with $N(u) = \set{r,s}$, where $\hh(r) = 3$ and $\hh(s) = 1$.
    Let $k := |V_2(T)|$.
    We proceed by induction on $k$.
    Since $T$ is a TD-unmixed balanced tree of height 3, $k$ is at least 2.
    
    First, suppose that $k = 2$.
    Then $T$ must be a balanced tree of the form shown in Figure~\ref{fig. basecase}, by Theorem~\ref{thm. constructing hh = 3 delt. trees}.
    \begin{figure}[ht]
        \centering
        \includegraphics[width = 0.5\textwidth]{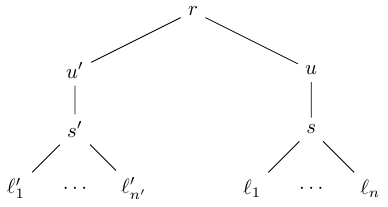}
        \caption{TD-unmixed balanced tree with two height 2 vertices}
        \label{fig. basecase}
    \end{figure}
    Using the vertex labeling in Figure~\ref{fig. basecase}, we have
    \begin{align*}
        I &= \calN_{odd}(T) = \gen{uu', u'\ell'_1\cdots \ell'_{n'}, u\ell_1\cdots \ell_{n}},\\
        C_{u,I} &= \gen{u', u'\ell'_1\cdots \ell'_{n'}, \ell_1\cdots \ell_{n}} = \gen{u',\ell_1\cdots \ell_{n}},\\
        N_{u,I} &= \gen{u'\ell'_1\cdots \ell'_{n'}}.
    \end{align*}
    One can verify that $I = C_{u,I} \cap (N_{u,I} + \gen{u})$. Furthermore, the ideals $N_{u,I}$ and $C_{u,I}$ are GVD by Lemma~\ref{lem. prod of var is GVD} and Corollary~\ref{cor. GVD SMI plus a var is GVD}, respectively. Thus, $I$ is GVD.

    Now suppose that $k > 2$.
    Let $T' := T\setminus \set{u}$ and $T'' := T\setminus N[u]$.
    By Lemma~\ref{lem. Cu,I and Nu,I}, we have $C_{u,I} = \calN_{odd}(T',T)$ and $N_{u,I} = \calN_{odd}(T'',T)$.
    By Theorem~\ref{thm. T - N[u] is delta forest}, $T''$ is a TD-unmixed balanced forest where the number of height 2 vertices in any connected component is less than $k$. Hence, $N_{u,I}$ is GVD by Theorem~\ref{theorem. I + J is GVD} and the induction hypothesis.

    Next, consider $\calN_{odd}(T',T)$.
    Since $\deg_T(u) = 2$, $T'$ has two connected components: $T_1$, which contains $r$, and $T_2$, which contains $s$.
    Hence we can write 
    $$
    \calN_{odd}(T',T) = \calN_{odd}(T_1,T) + \calN_{odd}(T_2, T).
    $$
    
    The connected component $T_2$ is isomorphic to a star graph if $|N_T(s) \cap V_0(T)| \geq 2$, which is a TD-unmixed balanced tree of height 1.
    Thus, $\calN_{odd}(T_2,T) = \calN_{odd}(T_2)$, which is GVD by Lemma~\ref{lem. height 0 and 1 are GVD}.
    If $N_T(s) \cap V_0(T) = \set{\ell}$ for some $\ell \in V_0(T)$, then $\calN_{odd}(T_2,T) = \gen{\ell}$, which is GVD by definition.
    
    If $\deg_T(r) > 2$, then $r$ is not a leaf in $T_1$. In this case, $T_1$ is a TD-unmixed balanced tree of height 3 (similar to the proof of Proposition~\ref{prop. T - r is delta forest}).
    Consequently, $\calN_{odd}(T_1,T) = \calN_{odd}(T_1)$, which is GVD by the induction hypothesis since $T_1$ has fewer than $k$ height 2 vertices.

    Now suppose that $\deg_T(r) = 2$.
    Let $u' \in V_2(T)$ be the vertex such that $N_T(r) = \set{u,u'}$.
    Let $T_1'$ be the connected component of $T\setminus \set{r}$ containing $u'$.
    Then $T_1$ is isomorphic to the graph obtained from $T_1'$ by attaching the leaf $r$ to $u'$.
    Thus, by setting $T_1'' := T_1'\setminus N_{T_1'}[u']$, we obtain
    \begin{align*}
        \calN_{odd}(T_1,T) &= \calN_{odd}(T_1',T) + \gen{u'}\\
        &= \gen{X_{N_{T_1'}(v)}\ :\ v \in (V(T_1') \cap V_{odd}(T))\setminus N_{T_1'}(u')} \tag{redundant}\\
        &\ \ \ + \gen{X_{N_{T_1'}(v)}\ :\ v \in N_{T_1'}(u')} + \gen{u'}\\
        &= \gen{X_{N_{T_1'}(v)}\ :\ v \in (V(T_1') \cap V_{odd}(T))\setminus N_{T_1'}(u')} + \gen{u'}\\
        &= \calN_{odd}(T_1'',T) + \gen{u'},
    \end{align*}
    where the second equality holds because the generators in the second term are divisible by $u'$ (corresponding to vertices adjacent to $u'$), making them redundant in the presence of $\gen{u'}$.
    Since $T_1''$ is obtained by removing $N[u']$ from $T_1'$, $T_1''$ is a TD-unmixed balanced forest by Theorem~\ref{thm. T - N[u] is delta forest}.
    Thus, $\calN_{odd}(T_1'',T) = \calN_{odd}(T_1'')$, which is GVD by the induction hypothesis. Consequently, $\calN_{odd}(T_1,T)$ is GVD by Corollary~\ref{cor. GVD SMI plus a var is GVD}.

    Finally, we verify the decomposition:
    \begin{align*}
        C_{u,I} \cap (N_{u,I} + \gen{u}) &= \left(\gen{X_{N_T(v)}\ :\ v \in V_{odd}(T) \setminus \set{r,s}} + \gen{\frac{X_{N_T(r)}}{u}, \frac{X_{N_T(s)}}{u}}\right)\\
        &\ \ \ \cap \left(\gen{X_{N_T(v)}\ :\ v \in V_{odd}(T) \setminus \set{r,s}} + \gen{u}\right)\\
        &= \gen{\lcm(X_{N_T(v)},X_{N_T(v')})\ :\ v,v' \in V_{odd}(T) \setminus \set{r,s}, v \neq v'}\\
        &\ \ \ + \gen{X_{N_T(v)}\ :\ v \in V_{odd}(T) \setminus \set{r,s}}\\
        &\ \ \ + \gen{u \cdot X_{N_T(v)}\ :\ v \in V_{odd}(T) \setminus \set{r,s}} + \gen{X_{N_T(s)},X_{N_T(r)}}\\
        &\ \ \ + \gen{\lcm\left(X_{N_T(v)},\frac{X_{N_T(v')}}{u}\right)\ :\ v \in V_{odd}(T)\setminus \set{r,s}, v' \in \set{r,s}}\\
        &= \gen{X_{N_T(v)}\ :\ v \in V_{odd}(T) \setminus \set{r,s}} + \gen{X_{N_T(s)},X_{N_T(r)}}\\
        &= I.
    \end{align*}
    Therefore, $I$ is GVD, and consequently, the even-stable complex $\calS_{even}(T)$ is vertex decomposable.
\end{proof}

\begin{remark}
    The condition $\deg_T(u) = 2$ for the choice of $u \in V_2(T)$ is not strictly necessary for the proof of Theorem~\ref{thm. oni is GVD}.
    We include this restriction to simplify the computation of the geometric vertex decomposition and to reduce the number of cases we must consider.
\end{remark}

\section{Open neighborhood ideals and facet ideals of simplicial trees}\label{sec. ONI and facet ideals}

In this section, we explore how open neighborhood ideals of TD-unmixed trees appear in other settings.
Previously, in \cite{COURAGE}, it was noted that every Cohen--Macaulay edge ideal of a tree is the odd-open neighborhood ideal of a TD-unmixed balanced tree.
In this work, we identify a broader class of ideals that properly contains the odd-open neighborhood ideals of TD-unmixed balanced trees.
We begin with a definition.

\begin{definition}\label{def. facet ideal}
    Let $\Delta$ be a simplicial complex on $V$.
    The \emph{facet ideal} of $\Delta$ is the square-free monomial ideal $\scrF(\Delta)$ in $\kk[V]$ generated by the facets of $\Delta$:
    $$
    \scrF(\Delta) := \gen{ X_{F} \ :\ F \in \calF(\Delta)}.
    $$
\end{definition}

Facet ideals were first introduced by Faridi in \cite{MR1935027}, viewing simplicial complexes as a generalization of graphs (treating $\Delta$ as a hypergraph with edge set $\calF(\Delta)$).
In this setting, a \emph{vertex cover} of $\Delta$ is a subset $S \subset V$ such that $S \cap F \neq \emptyset$ for all $F \in \calF(\Delta)$. We say $S$ is \emph{minimal} if no proper subset of $S$ is a vertex cover of $\Delta$; this is equivalent to the definition of a vertex cover for hypergraphs.
We say $\Delta$ is \emph{unmixed} if every minimal vertex cover of $\Delta$ has the same cardinality.
It follows that facet ideals admit the following primary decomposition.

\begin{theorem}
    Let $\Delta$ be a simplicial complex on $V$, and let $\VC(\Delta)$ and $\MVC(\Delta)$ be the sets of all vertex covers and minimal vertex covers of $\Delta$, respectively.
    Then
    $$
    \scrF(\Delta) = \bigcap_{S \in \VC(\Delta)} \gen{S} = \bigcap_{S \in \MVC(\Delta)}\gen{S},
    $$
    where the second decomposition is irredundant.
\end{theorem}

Faridi defined \emph{simplicial trees}, which are a generalization of tree graphs.

\begin{notation}
    Let $\Delta$ be a simplicial complex.
    We denote the simplicial complex obtained by removing a facet $F \in \calF(\Delta)$ by
    $$
    \Delta \setminus \gen{F} := \gen{F' : F' \in \calF(\Delta) \setminus \set{F}}.
    $$
    More generally, for a subset $S \subset \calF(\Delta)$, we define
    $$
    \Delta \setminus \gen{S} := \gen{F : F \in \calF(\Delta)\setminus S}.
    $$
\end{notation}

\begin{definition}\label{def. simp. trees}
    Let $\Delta$ be a simplicial complex.
    We say $\Delta$ is \emph{connected} if for any two vertices $u,v \in V$, there exists a sequence of facets $F_1,\dots,F_k \in \calF(\Delta)$ such that $u \in F_1$, $v \in F_k$, and $F_{i} \cap F_{i+1} \neq \varnothing$ for all $1 \leq i < k$.
    A facet $L \in \calF(\Delta)$ is a \emph{leaf} if $\Delta = \gen{L}$ (i.e., $\Delta$ is a simplex), or if there exists a facet $G \in \calF(\Delta) \setminus \set{L}$ such that for all $F \in \calF(\Delta) \setminus \set{L}$,
    $$
    L \cap F \subseteq L \cap G.
    $$
    The facet $G$ is called a \emph{joint} of $L$.
    A connected simplicial complex $\Delta$ is called a \emph{simplicial tree} if for any nonempty subset $S \subseteq \calF(\Delta)$, the simplicial complex $\Delta \setminus \gen{S}$ has a leaf.
    A non-connected simplicial complex $\Delta$ is a \emph{simplicial forest} if every connected component of $\Delta$ is a simplicial tree.
\end{definition}

In \cite{MR2121028}, Faridi provided a characterization of unmixed simplicial trees and showed that a simplicial tree is unmixed if and only if its facet ideal is Cohen--Macaulay.
We show that every odd-open neighborhood ideal of a TD-unmixed balanced tree can be identified with the facet ideal of a simplicial tree.
To do this, we utilize another equivalent definition of simplicial trees established by Cabora, Faridi, and Selinger \cite{MR2284286}.

\begin{definition}[\protect{\cite[Definition 3.1, Definition 3.4]{MR2284286}}]
    Let $\Delta$ be a non-empty simplicial complex.
    We say $\Delta$ is a \emph{cycle} if $\Delta$ has no leaves, but the subcomplex generated by every proper subset of $\calF(\Delta)$ has a leaf.
    Two distinct facets $F,G \in \calF(\Delta)$ are \emph{strong neighbors}, denoted $F \sim_{\Delta} G$, if for all $H \in \calF(\Delta)$, the condition $F \cap H \subseteq G$ implies $H = F$ or $H = G$.
\end{definition}

From the definitions above, it follows that a nonempty connected simplicial complex $\Delta$ is a simplicial tree if and only if it contains no cycles.
We utilize the following structural property of cycles in the proof of the main result of this section, Theorem~\ref{thm. odd ONI is facet ideal}.

\begin{theorem}[\protect{\cite[Proposition 3.11]{MR2284286}}]\label{thm. char of simp cycles}
    Let $\Delta$ be a simplicial complex that is a cycle, and let $n = |\calF(\Delta)|$.
    Then $n \geq 3$, and the facets of $\Delta$ can be enumerated as $\calF(\Delta) = \set{F_1,F_2,\dots,F_n}$ such that
    $$
    F_1 \sim_{\Delta} F_2 \sim_{\Delta} \cdots \sim_{\Delta} F_n \sim_{\Delta} F_1,
    $$
    and $F_i \not\sim_{\Delta} F_j$ for all other pairs of indices.
\end{theorem}

\begin{theorem}\label{thm. odd ONI is facet ideal}
    Let $T = (V,E)$ be a TD-unmixed balanced tree.
    Then there exists an unmixed simplicial tree $\Delta$ on $V_{even}$ such that $\calN_{odd}(T) = \scrF(\Delta)$.
\end{theorem}

\begin{proof}
    If $\hh(T) \leq 1$, then the claim is trivially true; in these cases, $\calN_{odd}(T)$ is the facet ideal of a simplex or the empty complex $\set{\varnothing}$, both of which are simplicial trees by definition.
    So suppose that $\hh(T) = 3$; recall that there are no TD-unmixed balanced trees of height 2 (see Example~\ref{ex. unmixed delta trees}).
    Let $\Delta$ be the simplicial complex on $V_{even}$ with facets
    $$
    \calF(\Delta) = \Set{N(v)}{v \in V_{odd}}.
    $$
    By the construction of $\Delta$, we have $\calN_{odd}(T) = \scrF(\Delta)$.
    It remains to show that $\Delta$ is a simplicial tree.
    Assume, for the sake of contradiction, that $\Delta$ is not a simplicial tree.
    Then $\Delta$ must contain a cycle $\Delta'$; consequently, $\calF(\Delta') \subseteq \calF(\Delta)$.
    By Theorem~\ref{thm. char of simp cycles}, we can enumerate the facets of $\Delta'$ as $\calF(\Delta') = \set{F_1,F_2,\dots,F_n}$ with $n \geq 3$ satisfying
    \begin{align*}
        F_1 \sim_{\Delta'} F_2 \sim_{\Delta'} \cdots \sim_{\Delta'} F_n \sim_{\Delta'} F_1. \tag{$\star$}
    \end{align*}
    By the construction of $\Delta$, for each $i$ with $1 \leq i \leq n$, there exists a vertex $v_i \in V_{odd}(T)$ such that $N(v_i) = F_i$.
    The condition $F_i \sim_{\Delta'} F_{i+1}$ (and $F_n \sim_{\Delta'} F_1$) implies that the facets intersect non-trivially. 
    Consequently, for each $i \in [n-1]$ (and similarly for the pair $v_n, v_1$), there is a path of length 2 between $v_i$ and $v_{i+1}$ in $T$ passing through a vertex in $F_i \cap F_{i+1} \subseteq V_{even}$.
    Thus, the sequence ($\star$) induces a cycle in $T$, contradicting the fact that $T$ is a tree.
\end{proof}

\begin{remark}
    Since $T$ is a tree, there are no two vertices of odd height $u,v \in V_{odd}$ such that $|N(u) \cap N(v)| \geq 2$.
    Thus, for any two facets $F_1$ and $F_2$ of the simplicial complex $\Delta$ constructed in Theorem~\ref{thm. odd ONI is facet ideal}, we have $|F_1 \cap F_2| \leq 1$.
    Consequently, one can construct an unmixed simplicial tree whose facet ideal cannot be the open neighborhood ideal of a balanced tree (for instance, see \cite[Example 7.5]{MR2121028}).
    Thus, we establish the following strict containment of classes of square-free Cohen--Macaulay ideals:
    $$
    \left\{ \parbox{3cm}{\centering CM edge ideals\\ of trees} \right\} \subsetneq \left\{ \parbox{4.2cm}{\centering CM odd-open neighborhood ideals\\ of balanced trees} \right\} \subsetneq \left\{ \parbox{3.5cm}{\centering CM facet ideals\\ of simplicial trees} \right\}.
    $$
\end{remark}

\begin{remark}
    It is established that the Stanley-Reisner complex associated with the edge ideal of a tree—viewed as a 1-dimensional simplicial tree—is vertex decomposable \cite{EdgeIdeal_VD}. 
    In contrast, it remains an open question whether the Stanley-Reisner complex of the facet ideal of a simplicial tree is always vertex decomposable. 
    In Theorem~\ref{thm. oni is GVD}, we demonstrate that a larger class of simplicial trees admits a vertex decomposable Stanley-Reisner complex.
\end{remark}

\section{Open neighborhood ideals of chordal graphs}\label{sec. ONI of chordal}

In this section, we examine the open neighborhood ideals of chordal graphs.
Recall that for $n \geq 3$, the $n$-\emph{cycle} is the graph $C_n$ with vertex set $V(C_n) = [n]$ and edge set $E(C_n) = \set{\set{1,2},\set{2,3},\dots,\set{n-1,n},\set{1,n}}$.
We say that a graph $G$ is a \emph{cycle graph} if $G$ is isomorphic to some $n$-cycle.
For simplicity, we will treat isomorphic graphs as equal.

\begin{definition}\label{def. chordal graph}
    Let $G = (V,E)$ be a graph.
    We say that a cycle graph $C$ is an \emph{induced cycle} of $G$ if there exists a subset $S \subseteq V$ such that $C = G[S]$.
    We say $G$ is \emph{chordal} if every induced cycle in $G$ is isomorphic to $C_3$.
\end{definition}

A constructive characterization of chordal graphs is given next.

\begin{theorem}[\cite{MR130190,MR270957}]\label{thm. const. char. of chordal graphs}
    Let $G$ be a connected graph on $n \geq 1$ vertices. Then $G$ is chordal if and only if $G$ can be constructed as follows:
    \begin{itemize}
        \item[(1)] Begin with a single vertex $v_1$.
        \item[(2)] For each $i$ from 2 to $n$, repeat the following process:\newline 
        (Step $i$) Add a new vertex $v_i$ and join it to a subset of vertices in $\set{v_1,\dots,v_{i-1}}$ that forms a clique in the graph constructed in Step $i-1$.
    \end{itemize}
\end{theorem}

For a finite set $V$ and a collection of subsets of $V$, say $\calA = \set{A_1,\dots,A_k} \subseteq 2^V$, we call $\calA$ a \emph{Sperner set} on $V$ if $A_i \not\subseteq A_j$ for all $i \neq j$.
A set $T \subseteq V$ is a \emph{transversal} of $\calA$ if $T \cap A_i \neq \varnothing$ for all $i$.
We say $T$ is \emph{minimal} if no proper subset of $T$ is a transversal of $\calA$.
Let $\tau(\calA)$ denote the set of all minimal transversals of $\calA$.
It is well known that $\tau(\calA)$ is also a Sperner set and that $\tau(\tau(\calA)) = \calA$ \cite[Chapter 2]{MR1013569}.
The following construction is due to Bahad\i r, Ekim, and G\"oz\"upek \cite{MR4333882}.

\begin{const}[Bahad\i r-Ekim-G\"oz\"upek, \protect{\cite[Proposition 2.3]{MR4333882}}]\label{const. any sqfree mon ideal is an ONI}
    Let $V$ be a finite set and let $\calA = \set{A_1,\dots,A_k} \subseteq 2^V$ be a Sperner set such that $V = \bigcup_{i = 1}^k A_i$ and $|A_i| > 1$ for all $i$.
    We construct the graph $G_{\calA}$ as follows:
    \begin{itemize}
        \item[(1)] Begin with a complete graph on the vertex set $V$; that is, add all possible edges between vertices in $V$.
        \item[(2)] Compute $\tau(\calA) = \set{T_1,\dots,T_m}$.
        \item[(3)] For each $T_i \in \tau(\calA)$, add a new vertex $t_i$ and add edges incident to $t_i$ such that $N(t_i) = T_i$.
    \end{itemize}
    Let $G_\calA$ with vertex set $V(G_\calA) = V \cup \set{t_1,\dots,t_m}$ be the resulting graph.
    Then the set of all minimal TD-sets of $G_{\calA}$ is exactly $\calA = \set{A_1,\dots,A_k}$.
\end{const}

\begin{example}\label{ex. BEG alg.}
    Let $V = \set{v_1,v_2,v_3,v_4,v_5}$ and
    $$
    \calA = \set{\set{v_1,v_2,v_3},\set{v_1,v_4}, \set{v_2,v_3,v_5}, \set{v_3,v_4,v_5}}.
    $$
    We construct the (chordal) graph $G_\calA$ whose minimal TD-sets are the sets in $\calA$ using Construction~\ref{const. any sqfree mon ideal is an ONI}.
    \begin{itemize}
        \item[(1)] Begin with a complete graph on the vertices $v_1,\dots,v_5$, as shown in Figure~\ref{fig: 11a}.
        \item[(2)] We have $\tau(\calA) = \set{\set{v_1,v_3},\set{v_1,v_5}, \set{v_2,v_3}, \set{v_2,v_4}, \set{v_3,v_4}}$. We set
        $$
            T_1 = \set{v_1,v_3}, \quad T_2 = \set{v_1,v_5}, \quad T_3 = \set{v_2,v_3}, \quad T_4 = \set{v_2,v_4}, \quad T_5 = \set{v_3,v_4}.
        $$
        \item[(3)] For each $T_i$, add a new vertex $t_i$ and connect it to the vertices in $T_i$, as shown in Figure~\ref{fig: 11b}.
    \end{itemize}
    The resulting graph in Figure~\ref{fig: 11b} is the graph $G_\calA$.
    \begin{figure}[ht]
    \centering
    \begin{subfigure}[b]{0.45\textwidth}
        \centering
        \includegraphics[width=\textwidth]{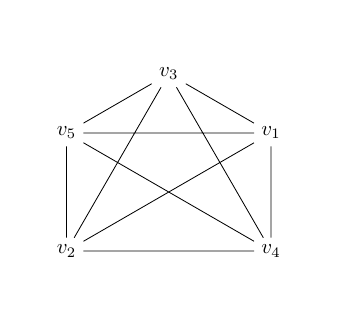}
        \caption{Complete graph on 5 vertices}
        \label{fig: 11a}
    \end{subfigure}
    \hfill 
    \begin{subfigure}[b]{0.45\textwidth}
        \centering
        \includegraphics[width = \textwidth]{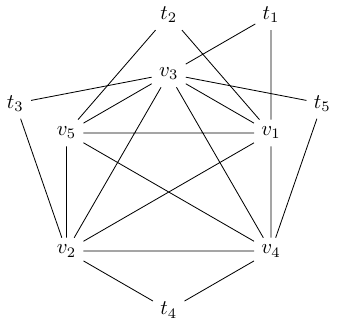}
        \caption{The graph $G_\calA$}
        \label{fig: 11b}
    \end{subfigure}
    \caption{Constructing $G_\calA$}
    \label{fig: 11}
\end{figure}
\end{example}

\begin{remark}
    In Step (1), one only needs to begin with a graph in which every vertex in $V$ is adjacent to at least one vertex in $A_i$ for all $i \in [k]$.
    However, we begin with a complete graph so that the resulting graph is chordal by Theorem~\ref{thm. const. char. of chordal graphs}.
    The condition $|A_i| > 1$ is necessary because a minimal TD-set cannot have size 1 in a simple graph (the open neighborhood of a vertex does not contain the vertex itself).
    If we wished to allow minimal TD-sets of size 1, we would need to permit loops (edges connecting a vertex to itself) in the graph, which would violate the assumption that the graph is simple.
\end{remark}

Next, we determine the minimal generating set of $\calN(G_{\calA})$ for a Sperner set $\calA$ satisfying the conditions in Construction~\ref{const. any sqfree mon ideal is an ONI}.

\begin{lemma}\label{lem. min. gen. of ONI of G_A}
    Let $V$, $\calA = \set{A_1,\dots,A_k}$, $\tau(\calA) = \set{T_1,\dots,T_m}$, and $G_{\calA}$ be as defined in Construction~\ref{const. any sqfree mon ideal is an ONI}.
    Then the minimal generating set of $\calN(G_\calA)$ is
    $$
    \Set{X_{N(t_i)}}{1 \leq i \leq m}.
    $$
\end{lemma}

\begin{proof}
    Since the open neighborhoods $N(t_i)$ in $G_\calA$ coincide with the sets $T_i \in \tau(\calA)$, which form a Sperner set, it suffices to show that the monomials $X_{N(v)}$ for all $v \in V$ are redundant.
    Let $v \in V$.
    By Step (1) of Construction~\ref{const. any sqfree mon ideal is an ONI}, $v$ is adjacent to every other vertex in $V$; thus, $X_{V\setminus \set{v}}$ divides $X_{N(v)}$.
    Hence, it remains to show that there exists some $t_i \in V(G_\calA)$ such that $X_{N(t_i)} \mid X_{V\setminus \set{v}}$ (i.e., $N(t_i) \subseteq V \setminus \set{v}$).

    Assume, for the sake of contradiction, that no such vertex $t_i$ exists.
    This implies that for every $i$, $N(t_i) \not\subseteq V \setminus \set{v}$.
    Since $N(t_i) = T_i \subseteq V$, it follows that $v \in T_i$ for all $i = 1, \dots, m$.
    Consequently, the singleton set $\set{v}$ is a transversal of $\tau(\calA) = \set{T_1, \dots, T_m}$.
    Since no proper subset of $\set{v}$ can be a transversal, $\set{v}$ is a minimal transversal.
    Therefore, $\set{v} \in \tau(\tau(\calA)) = \calA$, which contradicts the assumption that $|A_j| > 1$ for all $j \in [k]$.
\end{proof}

\begin{example}
    Consider the graph $G_\calA$ in Figure~\ref{fig: 11}.
    Since
    $$
    N(v_1) = \set{v_2,v_3,v_4,v_5,t_1,t_2} \supset \set{v_2,v_3} = N(t_3),
    $$
    we have $X_{N(t_3)} \mid X_{N(v_1)}$.
    Similarly, we have $X_{N(t_2)} \mid X_{N(v_2)}$, $X_{N(t_4)} \mid X_{N(v_3)}$, $X_{N(t_1)} \mid X_{N(v_4)}$, and $X_{N(t_5)} \mid X_{N(v_5)}$.
    Thus, we obtain
    $$
    \calN(G_\calA) = \gen{v_1v_3, v_1v_5, v_2v_3, v_2v_4, v_3v_4} \subset \kk[v_1,\dots,v_5,t_1,\dots,t_5].
    $$
\end{example}

By Lemma~\ref{lem. min. gen. of ONI of G_A} and Step (3) in Construction~\ref{const. any sqfree mon ideal is an ONI}, the indeterminates $t_1,\dots,t_m$ do not divide any minimal generator of $\calN(G_{\calA})$.
Thus, $t_1,\dots,t_m$ form a regular sequence in $\kk[V,t_1,\dots,t_m]/\calN(G_\calA)$.
This leads to the following corollary, which combines Theorem~\ref{theorem. prime decomp. of ONI}, Construction~\ref{const. any sqfree mon ideal is an ONI}, and Lemma~\ref{lem. min. gen. of ONI of G_A}.

\begin{corollary}\label{cor. any SFM as ONI of chordal}
    Let $I \subset R = \kk[x_1,\dots,x_n]$ be a square-free monomial ideal with the irredundant primary decomposition
    $$
    I = \bigcap_{i = 1}^k \gen{A_i},
    $$
    where $|A_i| > 1$ for all $i \in [k]$.
    Set $\calA = \set{A_1,\dots,A_k}$, $\tau(\calA) = \set{T_1,\dots,T_m}$, and $R' = R[t_1,\dots,t_m]$.
    Then we have an isomorphism
    $$
    R/I \cong R'/(\calN(G_\calA) + \gen{t_1,\dots,t_m}),
    $$
    where $G_\calA$ is the graph defined in Construction~\ref{const. any sqfree mon ideal is an ONI}.
\end{corollary}

\bibliographystyle{unsrt}

\end{document}